\documentclass[11pt,a4paper,reqno]{amsproc}
\usepackage{amssymb}
\usepackage{amsmath}

\usepackage{color}

\newtheorem*{theorem*}{Theorem}
\newtheorem*{theoremA*}{Theorem A}

\newtheorem{lemma}{Lemma}
\newtheorem*{lemma*}{Lemma}
\newtheorem{corollary}{Corollary}
\newtheorem*{corollary*}{Corollary}

\theoremstyle{definition}

\theoremstyle{remark}
\newtheorem{remark}{Remark}

\def\scal#1#2{\langle #1, #2\rangle}
\def\R#1{\mathbb{R}^{#1}}

\DeclareMathOperator{\Hess}{\mathrm{Hess}}

\begin{document}

\bigskip
\bigskip
\bigskip

\centerline{\textbf{A NOTE ON J\"ORGENS-CALABI-POGORELOV THEOREM\footnote{Translated from Soviet. Math. Dokl. (Communicated  by A.V. Pogorelov December 14, 1993), Vol.340, N.~3, p.317-318, see also \textbf{MR1328274 (96e:53100)}}}}

\bigskip
\centerline{\textbf{Tkachev Vladimir G.}}

\bigskip
\bigskip


%
%

1. Let $S_k(A)$ denote the $k$th principal symmetric function of the eigenfunctions of an $n\times n$ matrix $A$, i.e.
$$
\det (A+tI)=\sum_{k=0}^n S_k(A)t^{n-k}.
$$
The following classical result is well known.

\begin{theoremA*}[J\"orgens-Calabi-Pogorelov,  \cite{Jorg}, \cite{Ca1}, \cite{Pogor64}]
Let $f(x)$ be a convex entire solution  of
$$
S_n(\Hess f)\equiv \det (\Hess f)=1,\quad x\in \R{n},
$$
where $\Hess f$ is the Hessian matrix of $f(x)=f(x_1,\ldots,x_n)$. Then $f(x)$ is a quadratic polynomial, i.e.
\begin{equation}\label{eq1}
f(x)=a+\scal{b}{x}+\scal{x}{Ax},
\end{equation}
where $A$ is an $n\times n$ matrix with constant real coefficients and $\scal{}{}$ stands for the scalar product in $\R{n}$.
\end{theoremA*}

Let us consider the operator
\begin{equation}\label{eq2}
L[f]\equiv \sum_{i=1}^n a_i(x)S_i (\Hess f)=0.
\end{equation}
In  \cite{Borisenko}, A.A.~Borisenko has established that affine functions $f(x)=a+\scal{b}{x}$ are the only entire convex solutions of (\ref{eq2}) with the linear growth (i.e. $f(x)=O(\|x\|)$ as $x\to \infty$) in the following special cases, namely, when
\begin{equation}\label{eq3}
L[f]=S_n(\Hess f)-S_1(\Hess f)=\det \Hess f-\Delta f=0
\end{equation}
and
\begin{equation}\label{eq4}
L[f]=\sum_{k=0}^{[\frac{n-1}{2}]} (-1)^k S_{2k+1}(\Hess f)=0.
\end{equation}
Notice that solutions to (\ref{eq3}) and (\ref{eq4}) describe special  Lagrangian submanifolds given a non-parametric form.

Let us consider the following condition.

\begin{itemize}
\item[(Q)] either $a_k(x)\equiv 0$ on $\R{n}$, or there exist two positive constants $\mu_1\le \mu_2$ such that $\mu_1\le |a_k(x)|\le \mu_2$.
\end{itemize}

Let us denote by $J=J(L)$ the set of indices $i$, $1\le i\le n$ such that $a_i(x)\not\equiv 0$. The main purpose of this note is to establish the following generalization of \cite{Borisenko}.

\begin{theorem*}
Let $f(x)$ be an entire convex $C^2$-solution of (\ref{eq2}) and that the structural condition (Q) is satisfied. If
\begin{equation}\label{eq5}
\lim\sup_{\|x\|\to \infty}\frac{|f(x)|}{\|x\|^2}=0
\end{equation}
then $S_i (A(x))\equiv 0$ for any $i\in J$, in particular, $\det \Hess f(x)=0$. If additionally $a_1(x)\not \equiv 0$ then $f(x)$ is an affine function.
\end{theorem*}

\begin{remark}
We construct an example in paragraph~4 below which shows that (\ref{eq5}) is optimal in the sense that there exist operators $L$ satisfying the condition (Q) and possessing solutions growing quadratically $f(x)\sim \|x\|^2$ as $x\to \infty$ and such that $\Hess f(x)\not\equiv 0$.
\end{remark}

2. We use the standard convention to write $A\ge B$ if $A-B$ is a positive semi-definite matrix. 

\begin{lemma}\label{lem1}
Let $A(x)\ge 0$ be a continuous $n\times n$ matrix solution of
\begin{equation}\label{eq6}
L(A(x))\equiv \sum_{i=1}^n a_i(x)S_i (A(x))=0, \qquad x\in \R{n},
\end{equation}
where $L$ is subject to the condition (Q). Then either $S_i (A(x))\equiv 0$ for any $i\in J$, or there exists $k\in J$ and a constant $\sigma_0$ depending on $\mu_1$ and $\mu_2$ such that for all $x\in \R{n}$ the inequality holds
$$
S_k(A(x))\ge \sigma_0>0.
$$
\end{lemma}

\begin{proof}[Proof of Lemma~\ref{lem1}]
Note that $S_k(A(x))\ge0$ in virtue of the positive semi-definiteness of $A(x)$. Then, if all (non-identically zero) coefficients  $a_i$ have the same sign then $S_k(A(x))\equiv 0$ holds for any $i\in J$. Now suppose that there exists $x_0\in \R{n}$ and a number $k\in J$ such that $S_k(A(x_0))>0$. In that case, there exist  two coefficients $a_i$ having different signs. Observe that by the condition (Q) this also holds true in the whole $\R{n}$. Let us rewrite (\ref{eq6}) as
$$
|a_{i_1}(x_0)|S_{i_1}(A(x_0))+\ldots+|a_{i_m}(x_0)|S_{i_m}(A(x_0))=
|a_{j_1}(x_0)|S_{j_1}(A(x_0))+\ldots+|a_{j_p}(x_0)|S_{j_p}(A(x_0)),
$$
where $i_1<\ldots<i_m$, $j_1<\ldots< j_p$, and also $i_1<j_1$. We claim that $k=i_1$ satisfies the conclusion of the lemma. Indeed,  we have
\begin{equation}\label{eq7}
S_{i_1}(A(x_0))\le b_1S_{j_1}(A(x_0))+\ldots+b_p S_{j_p}(A(x_0)),
\end{equation}
where $b_k=|a_{j_k}(x_0)|/|a_{i_1}(x_0)|\le \mu_2/\mu_1$. Now, using Proposition~3.2.2 in \cite[p.~106]{Marcus}, we have
$$
\left(\frac{S_k(A(x_0))}{\binom{n}{k}}\right)^m\le \left(\frac{S_m(A(x_0))}{\binom{n}{m}}\right)^k,
$$
for any $1\le m\le k\le n$, therefore by (\ref{eq7})
$$
S_{i_1}(A(x_0))\le\frac{\mu_2}{\mu_1}\sum_{k=1}^{p}\alpha_k\cdot (S_{i_1}(A(x_0)))^{\nu_k},
$$
where $\nu_k=j_k/i_1>1$ and $\alpha_k=\binom{n}{j_k} \cdot \binom{n}{i_1}^{-\nu_k}$. Observe that the left hand side of the equation
$$
\frac{\mu_2}{\mu_1}\sum_{k=1}^{p}\alpha_k\cdot \sigma^{\nu_k-1}=1
$$
is an increasing function of $\sigma\ge0$, and let $\sigma=\sigma_0$ denote its (unique) positive root. Then in virtue of the positiveness of  $S_{i_1}(A(x_0))$ we conclude that $S_{i_1}(A(x_0))\ge \sigma_0$. By the continuity assumption on $A(x)$, the latter inequality also holds in the whole $\R{n}$ which proves the lemma.
\end{proof}

\begin{corollary}\label{cor1}
Let $f(x)\in C^2(\R{n})$ be a convex solution of $(\ref{eq2})$ under the condition (Q). Then  either $\det \Hess f\equiv 0$ in $\R{n}$ or there exists $k\in J$ such that the inequality
\begin{equation}\label{eq8}
S_k(\Hess f(x))\ge \sigma_0>0
\end{equation}
holds for all $x\in \R{n}$ with $k,\sigma_0$ chosen as in Lemma~\ref{lem1}.
\end{corollary}

3. \textbf{Proof of the Theorem}. We claim that under the hypotheses of the theorem there holds $S_i (\Hess f(x))\equiv 0$ for any $i\in J$. Indeed, arguing by contradiction we have by Lemma~\ref{lem1} that (\ref{eq8}) holds in the whole $\R{n}$ for some $k\in J$. One can assume without loss of generality, replacing if needed $f(x)$ by $f(x)+c+\scal{a}{x}$, that $f(x)\ge0$ in $\R{n}$. Given an arbitrary $\epsilon>0$, the condition (\ref{eq5}) yields the existence of  a constant $p\in \R{}$ such that $f(x)\le \frac{\epsilon}{2}\|x\|^2+p$ for any $x\in \R{n}$. But $g(x)=\frac{\epsilon}{2}\|x\|^2-f(x)\to \infty$ uniformly as $x\to \infty$, hence it attains its minimum value at some point, say $x_0\in \R{n}$, and there holds
$$
\Hess g(x_0)=\Hess (\frac{\epsilon}{2}\|x\|^2-f(x))|_{x_0}\ge0,
$$
which yields $\Hess f(x_0)\le \epsilon I$ with $I$ being the unit matrix.

Since $\Hess f(x_0)\ge 0$ we obtain applying the majorization principle (see, for instance Corollary~4.3.3 in \cite{Horn}) that
$$
S_k(\Hess f(x_0))\le S_k(\epsilon I)=\epsilon^k\binom{n}{k}.
$$
But the assumption $S_k(\Hess f(x))\ge \sigma_0$ yields  easily a contradiction with the arbitrariness of the $\epsilon$. This proves our claim. In particular, in virtue of the convexity of $f$ we also have $\Hess f(x)\ge 0$, hence $\Hess f(x)$ has zero eigenvalues for any $x\in \R{n}$ implying $\det \Hess f(x)\equiv 0$ in $\R{n}$ (see also  Corollary~\ref{cor1}). If, additionally, $a_1(x)\not\equiv 0$ then $1\in J$ and the claim implies   $S_1(\Hess f(x))=\Delta f(x)\equiv 0$. Applying again the convexity of $f(x)$ easily yields that $\Hess f(x)\equiv 0$ in $\R{n}$, hence $f(x)$ is an affine function and finishes the proof of the theorem.

4. E x a m p l e. Let $\alpha(t)$ be a positive function, non-identically constant and such that $0<q\le \alpha(t)\le q^{-1}$ for some fixed $0<q<1$. Let us consider the function
$$
f(x_1,\ldots,x_n)=\sum_{i=1}^n\int_{0}^{x_i}(x_i-t)\alpha(t)dt.
$$
Then $\Hess f(x)=(\alpha(x_i)\delta_{ij})_{1\le i,j\le n}$, hence  $f(x)$ is convex and satisfies
$$
S_n(\Hess f(x))-\omega(x)S_1(\Hess f)=0
$$
with $\omega(x)=\alpha(x_1)\ldots\alpha(x_n)/\sum_{i=1}^n\alpha(x_i)$. We have $\frac{1}{n}q^{n+1}\le a_1(x)\le \frac{1}{n}q^{-n-1}$, which establishes that  $L$ satisfies the condition (Q). On the other hand,
$$
\frac{q}{2}\|x\|^2\le f(x)\le \frac{1}{2q}\|x\|^2,
$$
thus $f(x)$ has the quadratic  growth at infinity.

\def\cprime{$'$}

\end{document}